\documentclass[12pt,reqno]{amsart}
\usepackage{amssymb,amsbsy,amsthm,amsmath,graphicx,epsfig}

\usepackage{setspace}
\usepackage{version}

\numberwithin{equation}{section}
\theoremstyle{plain}
\newtheorem{theorem}[subsection]{Theorem}
\newtheorem{proposition}[subsection]{Proposition}
\newtheorem{lemma}[subsection]{Lemma}
\newtheorem{corollary}[subsection]{Corollary}

\theoremstyle{definition}
\newtheorem{definition}[subsection]{Definition}

\theoremstyle{remark}

\renewcommand{\leq}{\leqslant}
\renewcommand{\geq}{\geqslant}
\newsavebox{\proofbox}
\savebox{\proofbox}{\begin{picture}(7,7)  \put(0,0){\framebox(7,7){}}\end{picture}}

\newcommand\Z{\mathbb{Z}}
\newcommand\R{\mathbb{R}}

\newcommand\SL{\operatorname{SL}}

\newcommand\GL{\operatorname{GL}}

\newcommand\id{{\operatorname{id}}}

\onehalfspace
\begin{document}

\title[A nilpotent dimension lemma]{A nilpotent Freiman dimension lemma}

\author{Emmanuel Breuillard}
\address{Laboratoire de Math\'ematiques\\
B\^atiment 425, Universit\'e Paris Sud 11\\
91405 Orsay\\
FRANCE}
\email{emmanuel.breuillard@math.u-psud.fr}

\author{Ben Green}
\address{Centre for Mathematical Sciences\\
Wilberforce Road\\
Cambridge CB3 0WA\\
England }
\email{b.j.green@dpmms.cam.ac.uk}

\author{Terence Tao}
\address{Department of Mathematics, UCLA\\
405 Hilgard Ave\\
Los Angeles CA 90095\\
USA}
\email{tao@math.ucla.edu}

\begin{abstract} We prove that a $K$-approximate subgroup of an arbitrary torsion-free nilpotent group can be covered by a bounded number of cosets of a nilpotent subgroup of bounded rank, where the bounds are explicit and depend only on $K$. The result can be seen as a nilpotent analogue to Freiman's dimension lemma.
\end{abstract}

\maketitle
%\tableofcontents

\begin{center}\emph{To the memory of Yahya Ould Hamidoune}\end{center}

\section{Introduction}

Freiman's Theorem (see e.g. \cite[Theorem 5.33]{tv-book}) asserts that any finite subset $A \subseteq \Z$ with \emph{doubling} $K$, that is to say with $|A+A|\leq K|A|$ for some parameter $K \geq 1$, is contained in a generalized arithmetic progression $P$ with rank $O_K(1)$ and size $O_K(1)|A|$. A number of papers have appeared in the past few years attempting to prove analogues of this result in groups other than the additive group of integers, and in particular in non-commutative groups. See, for example, \cite{breuillard-green,bg,breuillard-green-unitary,bgt,fisher-katz-peng,gill-helfgott,green-ruzsa,helfgott-sl2, helfgott-sl3,pyber-szabo,tao-noncommutative,tao-solvable}.

In \cite{bgt-structure}, we established a structure theorem for sets of small doubling and for \emph{approximate subgroups} of arbitrary groups. Approximate subgroups are symmetric finite subsets $A$ of an ambient group whose product set $AA$ can be covered by a bounded number of translates of the set. In particular they are sets of small doubling. They were introduced in \cite{tao-noncommutative}, because not only are they easier to handle than arbitrary sets of small doubling, but the study of arbitrary sets of small doubling reduces to a large extent to that of approximate subgroups. We refer the reader to Section \ref{facts} for a precise definition and a reminder of their basic properties.

One of the main theorems proved in \cite{bgt-structure} asserts, roughly speaking, that approximate subgroups can be covered by a bounded number of cosets of a certain finite-by-nilpotent subgroup with bounded complexity. The bounds on the number of cosets and on the complexity (rank and step) of the nilpotent subgroup depend only on the doubling parameter $K$. 

It turns out that there is an intimate analogy between approximate groups and neighborhoods of the identity in locally compact groups. In \cite{bgt-structure}, this analogy was worked out making use of ultrafilters in order to build a certain limit approximate group to which we then applied an adequately modified version of the tools developed in the 1950s for the solution of Hilbert's fifth problem on the structure of locally compact groups.

The use of ultrafilters makes the results of \cite{bgt-structure} ineffective in a key way. The aim of the present note is to give, in a fairly special case, a simple argument independent of \cite{bgt-structure} but which achieves the same goal with explicit bounds. Our main result is the following.

\begin{theorem}\label{freiman} Let $G$ be a simply-connected nilpotent Lie group and let $A$ be a $K$-approximate subgroup of $G$. Then $A$ can be covered by at most $\exp(K^{O(1)})$ cosets of a closed connected Lie subgroup of dimension at most $K^{O(1)}$.
\end{theorem}

The main feature of this theorem is that $G$ can be of arbitrary dimension and nilpotency class. One can be completely explicit about the bounds in this theorem as well as in the corollary below, and take $K^{2+29K^9}$ for the first bound on the number of cosets and $K^9$ for the bound on the dimension.

Theorem \ref{freiman} can be compared with Freiman's dimension lemma (see \cite[Theorem 5.20]{tv-book}), according to which a finite set with doubling at most $K$ in an arbitrary vector space over $\R$ is contained in an affine subspace with dimension bounded at most $\lfloor K-1 \rfloor$.

We note that \cite{bgt-structure} contains a result similar to Theorem \ref{freiman}, and even with the better bound $O(\log K)$ for the dimension of the nilpotent subgroup. Unlike Theorem \ref{freiman}, however, this provides no explicit bound whatsoever on the number of translates.

\begin{corollary}\label{cor} Let $G$ be a residually-\textup{(}torsion-free nilpotent\textup{)} group and let $A$ be a finite $K$-approximate subgroup of $G$. Then $A$ can be covered by at most $\exp(K^{O(1)})$ cosets of a nilpotent subgroup of nilpotency class at most $K^{O(1)}$.
\end{corollary}

\emph{Remark.} A group $G$ is  residually-(torsion-free nilpotent) if, for every $g \in G \setminus \{1\}$, there is a homomorphism $\pi : G \rightarrow Q$ where $Q$ is torsion-free nilpotent and $\pi(g) \neq \id$. This class of groups includes many non-nilpotent groups, such as all finitely generated free groups.

The main idea behind the proof of Theorem \ref{freiman} and Corollary \ref{cor} is encapsulated in a simple lemma, Lemma \ref{gleason} below. This lemma is essentially due to Gleason \cite{gleason3}, who uses an analogous idea in order to establish that the class of compact connected subgroups of a given locally compact group admits a maximal element.

The paper is organised as follows. In Section \ref{facts}, we recall the definition and some basic properties of approximate groups. In Section \ref{gleaglea}, we state and prove our key lemma. In Section \ref{commutator}, we establish that approximate subgroups of residually nilpotent groups have elements with a large centralizer. Finally in Section \ref{end}, we complete the proof of the main results.

\noindent\textsc{Acknowledgments.} EB is supported in part by the ERC starting grant 208091-GADA. He also acknowledges support from MSRI where part of this work was finalized. TT is supported by a grant from the MacArthur Foundation, by NSF grant DMS-0649473, and by the NSF Waterman award. We would like to thank T.~Sanders for a number of valuable discussions.

\setcounter{tocdepth}{1}

\section{Basic facts on approximate groups}\label{facts}

In this section we recall some basic properties of approximate groups. Background on approximate groups and their basic properties can be found in the third author's paper \cite{tao-noncommutative}.

\begin{definition}[Approximate groups]\label{approx-group-def}
Let $K \geq 1$. A finite subset $A$ of an ambient group $G$ is called \emph{$K$-approximate subgroup} of $G$ if the following properties hold:
\begin{enumerate}
\item the set $A$ is symmetric in the sense that $\id \in A$ and $a^{-1} \in A$ if $a \in A$;
\item there is a symmetric subset $X \subseteq A^3$ with $|X| \leq K$
such that $A \cdot A \subseteq X \cdot A$.
\end{enumerate}
\end{definition}

We record the following important, yet easy, fact.

\begin{lemma} \label{basic} Let $A$ be a $K$-approximate subgroup in an ambient group $G$ and $H$ a subgroup of $G$. Then we have the following facts.
\begin{enumerate}
\item $|A| \leq |A^2 \cap H| |AH/H| \leq |A^3|$.
\item $A^2 \cap H$ is a $K^3$-approximate subgroup of $G$. Moreover $|A^{k} \cap H| \leq K^{k-1} |A^2 \cap H|$ for all $k \geq 1$.
\item if $\pi : G \rightarrow Q$ is a homomorphism then $\pi(A)$ is a $K$-approximate subgroup of $Q$.
\end{enumerate}
\end{lemma}

\begin{proof} We only prove (ii), and leave the other items to the reader. We have $A^{k} \subseteq X^{k-1}A$. Note, however, that any two elements of $tA \cap H$ differ by right-multiplication by an element of $A^2 \cap H$.  It follows that $X^{k-1}A \cap H$ is contained in $Y(A^2 \cap H)$, where $|Y| \leq |X^{k-1}| \leq K^{k-1}$. The results follows.
\end{proof}

\section{A combinatorial analogue of a lemma of Gleason} \label{gleaglea}

In \cite{gleason3} Gleason established that every locally compact group admits a maximal compact connected subgroup. In order to prove that every increasing sequence of such subgroups must stabilize, he used a key lemma, \cite[Lemma 1]{gleason3}. Translated into the setting of approximate groups, this lemma reads as follows.
 
In what follows we denote by $\langle B \rangle$ the subgroup generated by the subset $B$.

\begin{lemma} \label{gleason} Let $G$ be an arbitrary group and let $\{\id\}=H_0 \subset H_1  \subset \dots \subset H_k$ be a nested sequence of subgroups of $G$. Let $A$ be a finite symmetric subset of $G$. Let $A_i:=A^2 \cap H_i$. Assume that $\langle A_{i+1} \rangle \nsubseteq A_{i+1}H_i$ for every $i=0,\dots,k-1$. Then $|A^5| \geq k|A|$.
\end{lemma}

\begin{proof} Clearly $A_{i+1}^2 \nsubseteq A_{i+1}H_i$, otherwise this would contradict the standing assumption. Pick $h_{i+1} \in A_{i+1}^2 \setminus A_{i+1}H_i$ for each $i=0,\dots,k-1$. We claim that the sets $Ah_i$ are disjoint. Indeed, if $Ah_{i+1} \cap Ah_j \neq \varnothing$ for some $j \leq i$ then $h_{i+1} \in A^2h_j \cap H_{i+1} = (A^2 \cap H_{i+1})h_j = A_{i+1} h_j\subseteq  A_{i+1}H_i$. This contradicts our assumption. 
\end{proof}

If $G$ is assumed to be a simply-connected nilpotent Lie group, then the above lemma can be refined in the following manner.

\begin{proposition} \label{nilpotent} Let $G$ be a simply-connected nilpotent Lie group and let $\{\id\}=H_0 \subset H_1  \subset \dots \subset H_k$ be a nested sequence of closed connected subgroups of $G$. Assume that, for each $i=0,\dots,k-1$, the set $A^2 \cap H_i$ is \emph{strictly} contained in $A^2 \cap H_{i+1}$. Then $|A^5| \geq k|A|$.
\end{proposition}

In order to prove this last proposition we first recall the following classical fact (see \cite[chap. 2]{raghunathan}).

\begin{lemma}\label{connected-closure} Let $G$ be a simply-connected nilpotent Lie group. For every subgroup $\Gamma$ in $G$, there exists a unique minimal closed connected subgroup $\overline{\Gamma}$ containing it.  Moreover, if $\Gamma_0$ has finite index in $\Gamma$, then $\overline{\Gamma}_0=\overline{\Gamma}$.
 %Let $H$ is a closed subgroup of $G$ and $H^0$ its connected component of the identity. Assume that $H/H^0$ is finite. Then $H=H^0$.
\end{lemma}

\begin{proof}[Proof of Proposition \ref{nilpotent}]. Let $A_i:=A^2 \cap H_i$. In view of Lemma \ref{gleason}, it is enough to prove that $\langle A_{i+1} \rangle \nsubseteq A_{i+1}H_i$ for each $i=0,\dots,k-1$.  Suppose, by contrast, that $\langle A_{i+1} \rangle \subseteq A_{i+1} H_i$.  Note that if $G_1, G_2$ are groups and $G_1 \subseteq X G_2$ with $X$ finite then $[G_1 : G_1 \cap G_2] < \infty$ (indeed we may assume that every element of $X$ lies in $G_1 G_2$ and hence, replacing $x \in X$ by $x g_2$ if necessary, that $X \subseteq G_1$; then $G_1 \subseteq X (G_1 \cap G_2)$). Our supposition therefore implies that  $\langle A_{i+1} \rangle \cap H_i$ has finite index in $\langle A_{i+1} \rangle$. Therefore, by Lemma \ref{connected-closure}, $\overline{\langle A_{i+1} \rangle} = \overline{ \langle A_{i+1} \rangle \cap H_i}\subset H_i$ and thus $A_{i+1}=A_i$. This is contrary to our assumption. 
\end{proof}

\section{An element with large centralizer}\label{commutator}

Another key part of our proof is to establish that any approximate subgroup has a large centraliser. This is, of course, a necessary step towards exhibiting the sought after nilpotent structure. 

This idea has intervened several times before under slightly different guises in the classification of approximate groups, be it in Helfgott's original paper \cite{helfgott-sl2}, in the first two authors' classification of approximate subgroups of compact Lie groups \cite{breuillard-green-unitary}, or in our recent paper on the structure of approximate groups in general \cite{bgt-structure}. It is also closely related to the key idea in the proof of the Margulis lemma in Riemannian geometry, or in the well-known geometric proof by Frobenius and Bieberbach of Jordan's theorem on finite linear groups.

\begin{proposition}\label{nilpotent-commutator} Let $G$ be a residually nilpotent group and let $A$ be a $K$-approximate subgroup of $G$. Then, there is an element $\gamma \in A^2$ such that $$|A^2 \cap C_G(\gamma)| \geq \frac{|A|}{K^6}.$$
\end{proposition}

\begin{proof} Let $C_1=G$ and $C_{i+1}=[G,C_i]$ be the central descending series of $G$. Let $k$ be the largest integer such that $A^2 \cap C_k \neq \{\id\}$. By the residually nilpotent assumption, and the fact that $A$ is a finite set, $k$ is finite. It follows that the map $\phi : A \times (A^6 \cap C_{k+1}) \rightarrow A^7$ defined by $\phi (a,x) = ax$ is injective. Indeed if $\phi(a,x) = \phi(a', x')$ then $ax=a'x'$, and so $a^{-1}a' \in A^2 \cap C_{k+1}=\{\id\}$, whence $a=a'$ and $x=x'$. As a consequence of this and the approximate group property of $A$ we get \[ |A| |A^6 \cap C_{k+1}| \leq |A^7| \leq K^6 |A|,\] and thus $$|A^6 \cap C_{k+1}| \leq K^6.$$

Now let $\gamma$ be any element of $A^2 \cap C_k$ other than the identity. If $a \in A$ then clearly $[\gamma,a] \in C_{k+1} \cap A^6$ and therefore $[\gamma,a]$ can take at most $K^6$ possible values as $a$ varies in $A$. Let $[\gamma,x]$ be the most popular such value. We have $[\gamma,x]=[\gamma,y]$ for at least $K^{-6}|A|$ elements $x,y \in A$. For any two such $x,y$ a short calculation confirms that $x^{-1}y \in C_G(\gamma)$. We conclude $|A^2 \cap C_G(\gamma)| \geq K^{-6}|A|$, which is what we wished to prove.
\end{proof}

\section{Proof of Theorem \ref{freiman} and Corollary \ref{cor}}\label{end}

We first prove Theorem \ref{freiman}. We will build inductively two nested sequences  $$G=G_0 \supset G_1 \supset  \dots \supset G_k$$ and $$\{1\}=H_0 \subset H_1 \subset \dots \subset H_k $$ of connected closed subgroups of $G$, together with elements $\gamma_{i}\in A^4$ such that, for each $i=1,\dots,k$:
\begin{enumerate}
\item each $H_i$ is a normal subgroup of $G_i$;
\item $\gamma_{i}$ normalizes $H_{i-1}$ and $H_{i}=\overline{\langle H_{i-1},\gamma_{i}\rangle}$ is the connected subgroup generated by $H_{i-1}$ and $\gamma_{i}$;
\item $\gamma_{i} \in  H_i\setminus H_{i-1}$;
\item $|A^2 \cap G_i| \geq K^{-n_i}|A|$, where $n_i$ is a increasing sequence of integers to be determined later.
\end{enumerate}

Moreover, assuming that such data has been built up to level $k$, we will be able to build a new level $(k+1)$ provided that $A^2 \cap G_k \nsubseteq H_k$.

Before establishing the existence of the above data, let us see how this concludes the proof of Theorem \ref{freiman}. From (ii) and (iii) we see that $\dim H_{i}=\dim H_{i-1} +1$, and hence $\dim H_i=i$ for all $i \leq k$. Since $G$ is finite dimensional, this implies that the process must stop at some finite time $k$, at which point we must have $A^2 \cap G_k \subseteq H_k$. However, (iii) implies that $\gamma_i \in A^4 \cap (H_{i} \setminus H_{i-1})$ and thus  $A^4 \cap H_{i-1}$ is strictly contained in $A^4 \cap H_{i}$ for all $i=1,\dots,k$. From Proposition \ref{nilpotent}, we conclude that $|A^{10}| \geq k |A^2|$. Since $A$ is a $K$-approximate group, $|A^{10}| \leq K^9|A|$ and hence $k \leq K^9$. However, (iv) implies that $|A^2 \cap G_k| \geq K^{-n_k}|A|$. Since $A^2 \cap  G_k \subseteq H_k$ it follows that  $|A^2 \cap H_k| \geq K^{-n_k}|A|$ and we conclude by Lemma \ref{basic}(i) that $A$ can be covered by at most $K^{2+n_k}$ translates of the closed subgroup $H_k$, which has dimension $k \leq K^9$. An explicit value for $n_i$ is computed below; it gives the desired bound in Theorem \ref{freiman} and this concludes the proof.

It remains to establish that the data above can indeed be constructed. This will be done by induction on $k$; assuming that the data has been built up to some level $k \geq 0$ and that $A^2 \cap G_k \nsubseteq H_k$, we will build the $(k+1)$-st level.

Let $A'_k:=(A^2 \cap G_k)/H_k$. By our assumption $A'_k \neq \{\id\}$, and so from Lemma \ref{basic} we see that $A'_k$ is a $K^3$-approximate subgroup of $G_k/H_k$. Since $G_k/H_k$ is a simply connected nilpotent Lie group, we may apply Proposition \ref{nilpotent-commutator} to it and conclude that there exists some element $\gamma_{k+1}' \in {A_{k}^{\prime}}^2\setminus \{\id\}$ such that
\begin{equation}\label{commbound}
|{A_{k}^{\prime}}^2 \cap C_{G_k/H_k}(\gamma_{k+1}')| \geq K^{-18}|A_k'|.
\end{equation}
Note that $C_{G_k/H_k}(\gamma_{k+1}')$ is a closed connected subgroup of $G_k/H_k$, because in a simply-connected nilpotent Lie group the centraliser of an element coincides with the analytic subgroup whose Lie algebra is the centraliser of the logarithm of that element. We can then define $G_{k+1}$ to be the closed connected subgroup of $G_k$ such that $C_{G_k/H_k}(\gamma_{k+1}')=G_{k+1}/H_k$.

We can also lift $\gamma_{k+1}'$ to an element $\gamma_{k+1} \in (A^2 \cap G_k)^2$ so that $\gamma_{k+1}H_k=\gamma_{k+1}'$. Note that $\gamma_{k+1}
\notin H_k$ and that $\gamma_{k+1} \in G_{k+1}$.

We then define $H_{k+1}$ to be the closed connected subgroup of $G_{k+1}$ generated by $H_k$ and $\gamma_{k+1}$. Since $G_{k+1}$ commutes with $\gamma_{k+1}$ modulo $H_k$, we conclude that $H_{k+1}$ is a normal subgroup of $G_{k+1}$. This shows the first three items (i), (ii) and (iii).

It remains to establish (iv). From \eqref{commbound} we obtain \begin{equation}\label{newcommbound}|(A^4 \cap G_{k+1})/H_k| \geq K^{-18}|A'_k|,\end{equation} and from Lemma \ref{basic} (i) we have $$K^{11}|A^2 \cap G_{k+1}| \geq |A^{12} \cap G_{k+1}| \geq |(A^4 \cap G_{k+1})^3| \geq |(A^4 \cap G_{k+1})/H_k||A^4 \cap H_k|.$$ Combining this last line with \eqref{newcommbound} yields $$K^{11}|A^2 \cap G_{k+1}| \geq K^{-18}|(A^2 \cap G_k)/H_k||A^2 \cap H_k| \geq K^{-18}|A^2 \cap G_k|,$$ where the last inequality follows from Lemma \ref{basic} (i). Therefore $$|A^2 \cap G_{k+1}| \geq K^{-29}|A^2 \cap G_k|  \geq K^{-n_k - 29}|A|.$$ We thus see that we can set $n_{k+1}=n_k+29$, or alternatively $n_i=29i$ for all $i$. This ends the proof of Theorem \ref{freiman}.

\begin{proof}[Proof of Corollary \ref{cor}] Suppose first that $G$ is torsion-free nilpotent. Without loss of generality, we may assume that $G$ is generated by the finite set $A$. We may then consider the Malcev closure $\widetilde{G}$ of $G$ over $\R$ (see \cite[Chap. 2]{raghunathan}) and apply Theorem \ref{freiman} to conclude that $A$ is covered by $\exp(K^{O(1)})$ cosets of a closed connected subgroup $L$ of $\widetilde{G}$ with dimension at most $K^{O(1)}$. One of these cosets intersects $A$ in more than $\exp(-K^{O(1)})|A|$ elements, from which it follows that $|A^2 \cap L| \geq \exp(-K^{O(1)})|A|$. We let $H$ be the subgroup generated by $A^2 \cap L$. It is a finitely generated torsion-free nilpotent group with nilpotency class (and indeed length of any composition series) at most the dimension of $L$, that is to say at most $K^{9}$. Finally, by Lemma \ref{basic} (i), we see that $A$ can be covered by at most $\exp(K^{O(1)})$ cosets of $H$, as desired.

Now if $G$ is only assumed to be residually torsion-free nilpotent, then for every constant $M\geq 1$, there exists a torsion-free nilpotent quotient $G/\ker \pi$ of $G$ such that $A^{M} \cap\ker \pi = \{\id\}$. If $M$ is taken larger than the word length of every commutator of length $K^9$, then all commutators of length $K^9$ in the elements of $\pi^{-1}(\pi(A)^2 \cap L)$ are trivial, and thus the subgroup $H$ of $G$ generated by $\pi^{-1}(\pi(A)^2 \cap L)$ is nilpotent with nilpotency class at most $K^9$. By Lemma \ref{basic} (i), we see that $A$ intersects only at most $\exp(K^{O(1)})$ cosets of $H$ and we are done.
\end{proof}

\end{document}